\documentclass{amsart}
\usepackage{amsmath}
\usepackage{amscd}
\usepackage{amssymb}

\newtheorem{theorem}{Theorem}[section]
\newtheorem{theorem/definition}{Theorem/Definition}[section]
\newtheorem{proposition}{Proposition}[section]
\newtheorem{lemma}{Lemma}[section]

\theoremstyle{remark}
\newtheorem{remark}{Remark}[section]
\theoremstyle{definition}

\newtheorem{example}{Example}[section]

\begin{document}
\title
{On Locally Conformally Flat gradient Steady Ricci Solitons}
\author{Huai-Dong Cao and Qiang Chen}
\address{Department of Mathematics\\ Lehigh University\\
Bethlehem, PA 18015} \email{huc2@lehigh.edu; qic208@lehigh.edu}

\begin{abstract}
In this paper, we classify $n$-dimensional ($n \geq 3$) complete
noncompact locally conformally flat gradient steady solitons. In
particular, we prove that a complete noncompact non-flat
locally conformally flat gradient steady Ricci soliton is, up to scaling,
the Bryant soliton.

\end{abstract}

\maketitle
\date{}

\footnotetext[0]{2000 {\sl Mathematics Subject ClassiÞcation}. Primary 53C21, 53C25. 

The first author was partially supported by NSF Grants  DMS-0506084 and DMS-0909581; 
the second author was partially supported by NSF Grant DMS-0354621 and a Dean's Fellowship of
the School of Arts and Sciences at Lehigh University.}

\section{The results}

A complete Riemannian metric $g_{ij}$ on a smooth manifold $M^n$
is called a {\it gradient steady Ricci soliton} if there exists
a smooth function $F$ on $M^n$ such that the Ricci tensor $R_{ij}$
of the metric $g_{ij}$ is given by the Hessian of $F$:
$$R_{ij}=\nabla_i\nabla_jF. \eqno(1.1)$$
The function $F$ is called a {\it potential function} of the gradient steady soliton.
Clearly, when $F$ is a constant the gradient steady Ricci soliton is simply a Ricci flat manifold. 
Thus Ricci solitons are natural extensions of Einstein metrics. 
Gradient steady solitons play an important role in
Hamilton's Ricci flow as they correspond to translating
solutions, and often arise as Type II singularity models. Thus one is interested
in classifying them and understand their geometry.

It turns out that compact gradient steady solitons must be Ricci flat. In dimension $n=2$, Hamilton \cite{Ha88} discovered the 
first example of a complete noncompact gradient steady
soliton on $\mathbb R^2$, called the {\it cigar soliton}, where the metric is given by
$$ ds^2=\frac{dx^2 +dy^2}{1+x^2+y^2}.$$ 
The cigar soliton has positive curvature and is asymptotic to a cylinder of finite
circumference at infinity.  Furthermore, Hamilton \cite{Ha88} showed that {\sl the only
complete steady soliton on a two-dimensional manifold with
bounded (scalar) curvature $R$ which assumes its maximum
at an origin is, up to scaling,  the cigar soliton}. For $n\geq 3$, Bryant \cite{Bryant} proved that
{\sl there exists, up to scaling, a unique complete rotationally symmetric gradient Ricci
soliton on $\Bbb R^n$}, see, e.g.,  Chow et al. \cite{Chow et al 1}
for details. The Bryant soliton has positive sectional curvature, linear curvature decay
and volume growth on the order of $r^{(n+1)/2}$. In the K\"ahler case,
the first author \cite{Cao94} constructed a complete gradient steady K\"ahler-Ricci soliton on $\mathbb{C}^m$,
for $m\geq 2$, with positive sectional curvature and $U(m)$ symmetry.

A well-known conjecture is that, in dimension $n=3$, the Bryant soliton is the
only complete noncompact ($\kappa$-noncollapsed) gradient steady
soliton with positive sectional curvature\footnote{Perelman (\cite{P1}, {\bf 11.9})
claimed that the conjecture is true but didn't give any detail, or sketch,  of a
proof.}.  We remark that S.-C.
Chu \cite{Chu} and H. Guo \cite{Guo} have studied the geometry of
3-dimensional gradient steady solitons with positive sectional
curvature and the scalar curvature $R$ attaining its maximum at
some origin. For $n\ge 4$, it is also natural to ask if the Bryant
soliton is the only complete noncompact, positively curved,
locally conformally flat gradient steady soliton. In this paper,
we classify $n$-dimensional ($n\geq 3$) complete noncompact
locally conformally flat gradient steady solitons, and give an
affirmative answer to the latter question.   Our main results:

\begin{theorem} Let $(M^n, g_{ij}, F)$, $n\ge 3$, be a $n$-dimensional  complete noncompact locally conformally flat
gradient  steady Ricci soliton with positive sectional curvature. Then,  $(M^n, g_{ij}, F)$ is isometric to 
the Bryant soliton.
\end{theorem}

\begin{theorem} Let $(M^n, g_{ij}, F)$, $n\ge 3$, be a $n$-dimensional  complete noncompact locally conformally flat
gradient  steady Ricci soliton. Then,  $(M^n, g_{ij}, F)$ is
either flat  or isometric to the Bryant soliton.
\end{theorem}

Our work was motivated in part by the works of physicists Israel \cite{Is} and Robinson \cite{Robin} concerning
the uniqueness of the Schwarzschild black hole among all static,
asymptotically flat vacuum space-times. In their
setting, the Einstein field equations take the form $$R_{ij}+
V^{-1}\nabla_i\nabla_jV =0 \eqno(1.2)$$ and $$ \Delta V=0 $$ for certain positive potential
function $V$ on a three-dimensional space-like hypersurface $(N^3,
g_{ij})$.   They proved that such $(N^3, g_{ij}, V)$ must be rotationally symmetric without assuming
locally conformal flatness. In fact, they were able to prove such $(N^3, g_{ij}, V)$ is locally conformally 
flat. However, it remains a challenge to do the same for 3-dimensional gradient steady Ricci solitons. 

\medskip \noindent {\bf Acknowledgments}. The first author is grateful to Rick Schoen for bringing the papers
 \cite{Is, Robin} to his attention.
He also would like to thank Bing-Long Chen and Xi-Ping Zhu for
helpful discussions, and S.-C. Chu for pointing out the reference
\cite{Guo}.

\section{Preliminaries}

In this section, we recall some basic facts and collect several known results about gradient steady
solitons.

Let $(M^n, g_{ij}, F)$ be a gradient steady Ricci soliton so that the Ricci tensor $R_{ij}$
of the metric $g_{ij}$ is given by the Hessian of the potential function $F$: 
$$R_{ij}=\nabla_i\nabla_jF.$$
Taking the covariant derivatives and
using the commutating formula for covariant derivatives, we obtain
$$\nabla_iR_{jk}-\nabla_jR_{ik}=R_{ijkl}\nabla_lF. \eqno(2.1)$$
Taking the trace on $j$ and $k$, and using the contracted second
Bianchi identity
$$\nabla_j R_{ij}=\frac {1}{2}\nabla_i R,$$
we get
$$\nabla_iR=-2R_{ij}\nabla_jF. $$
Thus
$$\nabla_i(R+|\nabla F|^2)=-2(R_{ij}-\nabla_i\nabla_jF)\nabla F=0.$$
Hence  we have 

\begin{lemma} {\bf (Hamilton \cite{Ha95F})} Let $(M^n, g_{ij}, F)$
be a complete gradient stead soliton satisfying Eq. (1.1).
Then we have
$$\nabla_iR=-2R_{ij}\nabla_jF, \eqno(2.2)$$ and
$$R+|\nabla F|^2=C_0 \eqno(2.3)$$ for some constant $C_0$. Here $R$
denotes the scalar curvature.
\end{lemma}

Taking the trace in (1.1), we get

$$R=\Delta F. \eqno(2.4)$$

Combining  (2.4) with (2.3), we have

$$\Delta F +|\nabla F|^2=C_0. \eqno(2.5)$$

If $M$ is compact, then it follows that 
$$0=\int_{M} \Delta (e^{F}) =\int_{M} (\Delta F +|\nabla F|^2) e{^F}=C_0 \int_M e^F.$$
Thus $C_0=0$. Now integrating (2.5), with  $C_0=0$, over $M$  yields
$$\int_M |\nabla F|^2=0,$$ implying that $F$ is a constant.  
Therefore, we get 

\begin{proposition}  {\bf (cf. Hamilton \cite{Ha95F}, Ivey \cite{Iv})} On a compact manifold
$M^n$, a gradient steady  Ricci soliton is necessarily a Ricci flat Einstein metric. 
\end{proposition}

 Complete noncompact steady solitons do exist and here are some basic examples: 

\begin{example}{\bf (The cigar soliton)}
In dimension two, Hamilton \cite{Ha88} discovered the first
example of a complete noncompact steady soliton on $\mathbb R^2$,
called the {\it cigar soliton}, where the metric is given by
$$ ds^2=\frac {dx^2 +dy^2} {1+x^2+y^2}$$ with potential function
$$F=\log (1+x^2+y^2).$$
The (scalar) curvature of the cigar soliton is given by 
$$R=\frac{1}{1+x^2+y^2}.$$ Hence it is positive, attains its maximum at the origin, and decays to zero exponentially fast
(in terms of the geodesic distance) at space infinity. Furthermore, the cigar soliton has linear volume
growth, and is asymptotic to a cylinder of finite circumference at
$\infty$.
\end{example}

\begin{example}{\bf (The Bryant soliton)}
In the Riemannian case, higher dimensional examples of noncompact
gradient steady solitons were found by Robert Bryant \cite{Bryant} on
$\Bbb R^n$ ($n\geq 3$). They are rotationally symmetric and have
positive sectional curvature. Furthermore, the geodesic sphere
$S^{n-1}$ of radius $s$ has the diameter on the order $\sqrt{s}$.
Thus the volume of geodesic balls $B_r(0)$ grow on the order of
$r^{(n+1)/2}$.
\end{example}

\begin{example}{\bf (Noncompact gradient steady K\"ahler solitons)}
In the K\"ahler case, the first author \cite{Cao94} found two examples
of complete rotationally noncompact gradient steady K\"ahler-Ricci
solitons

(a) On $\Bbb C^n$ (for $n=1$ it is just the cigar soliton). These
examples are $U(n)$ invariant and have positive sectional
curvature. It is interesting to point out that the geodesic sphere
$S^{2n-1}$ of radius $s$ is an $S^1$-bundle over $\Bbb CP^{n-1}$
where the diameter of $S^1$ is on the order $1$, while the
diameter of $\Bbb CP^{n-1}$ is on the order $\sqrt{s}$. Thus the
volume of geodesic balls $B_r(0)$ grow on the order of $r^{n}$,
$n$ being the complex dimension. Also, the curvature $R(x)$ decays
like $1/r$.

(b) On the blow-up of $\Bbb C^n/\Bbb Z_n$ at the origin. This is
the same space on which Eguchi-Hansen  ($n=2$) and Calabi
 ($n\geq 2$) constructed examples of Hyper-K\"ahler
metrics. For $n=2$, the underlying space is the canonical line
bundle over $\mathbb{CP}^1$.
\end{example}

Now a Ricci flat metric is clearly a stationary  solution of Hamilton's Ricci flow
$$\frac{\partial g_{ij}(t)}{\partial t}=-2R_{ij}(t).$$ This happens, for example, on a flat torus or on any
$K3$-surface with a Calabi-Yau metric.

On the other hand, suppose that we have a complete steady Ricci soliton $g_{ij}$ on a smooth manifold $M^n$ 
with potential function $F$. As observed by  Z.-H.
Zhang \cite{Zh1}, the gradient vector field $V=\nabla F$ is a complete vector field on
$M$. Let $\varphi_t$ denote the one-parameter group
of diffeomorphisms  of $M^n$ generated by $-V$. Then it easily follows that 
$$\tilde{g}_{ij}(t)=\varphi^*_t g_{ij} $$ is a solution to the Ricci flow for $ -\infty<t<\infty$, with $\tilde{g}_{ij}(0)=g_{ij}$.

Next we present a useful result, which was implicitly proved by B.-L. Chen \cite{BChen}.  For the reader's convenience,
we include a proof here (see also Proposition 5.5 in \cite{Cao08b}). 

\begin{proposition}  Let $g_{ij}(t)$ be a complete ancient
solution to the Ricci flow on a noncompact manifold $M^n$. Then the scalar curvature
$R$ of $g_{ij}(t)$ is nonnegative for all $t$.
\end{proposition}

\begin{proof} Suppose $g_{ij}(t)$ is defined for $-\infty <t\leq T$
for some $T>0$. We divide the argument into
two steps:

Step 1: Consider any complete solution $g_{ij}(t)$ defined on $[0,
T]$. For any fixed point $x_0\in M$, pick $r_0> 0$ sufficiently
small so that
$$|Rc|(\cdot, t)\leq (n-1) r_0^{-2} \qquad \mbox{on} \ B_t(x_0, r_0)$$ for
all $t\in [0, T]$. Then for any positive number $A>2$, pick
$K_A>0$ such that $R\geq -K_A$ on $B_0(x_0, Ar_0)$ at $t=0$. We
claim that there exists a universal constant $C>0$ (depending on
the dimension $n$) such that
$$R(\cdot, t)\geq \min\{-\frac{n} {t+\frac{1}{K_A}},
-\frac{C}{Ar_0^2}\} \qquad \mbox{on} \ B_t(x_0, \frac{3A}{4}r_0)
\eqno(2.6)$$ for each $t\in [0,T]$.

Indeed, take a smooth nonnegative decreasing function $\phi$ on
$\Bbb R$ such that $\phi =1$ on $(-\infty,7/8]$, and $\phi=0$ on
$[1, \infty)$. Consider the function $$u(x, t)=\phi(\frac{d_t(x_0,
x)} {Ar_0}) R(x, t),$$ where $d_t(x_0, x)$ is the distance function from $x_0$ to $x$ at time $t$. Then we have
$$(\frac{\partial}{\partial t}-\Delta)u=\frac{\phi' R}{Ar_0}
(\frac{\partial}{\partial t}-\Delta)d_t(x_0,
x)-\frac{\phi''R}{(Ar_0)^2} +2\phi |Rc|^2-2\nabla \phi\cdot\nabla
R$$ at smooth points of  $d_t(x_0, \cdot)$.

Let $u_{\min}(t)=\min_{M} u(\cdot, t)$. Whenever
$u_{\min}(t_0)\leq 0$, assume $u_{\min}(t_0)$ is achieved at some
point $\bar x\in B_{t_0}(x_0, r_0)$, then $\phi' R(\bar x,
t_0)\geq 0$. On the other hand, by Lemma 8.3(a) of Perelman
\cite{P1} (see also Lemma 3.4.1 (i) of \cite{CaoZhu}), we know that
$$(\frac{\partial}{\partial t}-\Delta)d_t(x_0, x)\geq -\frac{5(n-1)}{3r_0}$$
outside $B_t(x_0, r_0)$. Following (Section 3, Hamilton
\cite{Ha86}), we define
$$\left.\frac{d}{dt}\right|_{t=t_0} u_{\min}=\liminf_{h\to 0^+}
\frac{u_{\min}(t_0+h)-u_{\min}(t_0)}{h},$$ the liminf of all
forward difference quotients, then

$$ \left.\frac{d}{dt}\right|_{t=t_0}
u_{\min}\geq -\frac{5(n-1)}{3Ar_0^2}\phi' R+ \frac{2}{n}\phi R^2 +
\frac{1}{(Ar_0)^2} (\frac{2{\phi'}^2}{\phi}-\phi'')R .$$ Hence,

$$\left.\frac{d}{dt}\right|_{t=t_0} u_{\min} \geq \frac{1}{n}u_{\min}^{2}(t_0)
-\frac{C^2}{(Ar_0^2)^2}$$ provided $u_{\min}(t_0)\leq 0$. Now
integrating the above inequality, we get
$$u_{\min}(t)\geq \min\{-\frac{n} {t+\frac{1}{K_A}},
-\frac{C}{Ar_0^2}\} \qquad \mbox{on} \ B_t(x_0, \frac{3A}{4}r_0),
$$ and the inequality (2.6) in our claim follows.

Step 2: Now if our solution $g_{ij}(t)$ is ancient, we can replace
$t$ by $t-\alpha$ in (2.6) and get
$$R(\cdot, t)\geq \min\{-\frac{n} {t-\alpha +\frac{1}{K_A}},
-\frac{C}{Ar_0^2}\} \qquad \mbox{on} \ B_t(x_0, \frac{3A}{4}r_0).$$ Letting $A\to \infty$ and then $\alpha \to -\infty$,
we see that $R(\cdot, t)\ge 0$ for all $t$. This completes the proof of Proposition 2.2.

\end{proof}

As an immediate corollary, we have 

\begin{lemma}
Let $(M^n, g_{ij}, F)$ be a complete gradient steady soliton. Then it has nonnegative
scalar curvature $R\ge 0$.
\end{lemma}

We remark that Munteanu-Sesum \cite{MS} recently proved a Liouville type theorem for gradient steady solitons, namely a 
gradient steady soliton does not admit any non-trivial harmonic
function with finite Dirichlet energy. As a consequence, a gradient steady soliton has at most one nonparabolic end. 

For steady solitons with nonnegative Ricci curvature, we have the following 

\begin{lemma} {\bf (Hamilton \cite{Ha95F})}
Let $(M^n, g_{ij}, F)$ be a complete gradient steady soliton with nonnegative Ricci curvature $Rc\ge 0$ and assume
the scalar curvature $R$ attains its maximum at some point $x_0$. Then the potential function $F$ is weakly convex
and attains  its minimum at $x_0$.
\end{lemma}

Moreover, if  the Ricci curvature of $(M^n, g_{ij}, F)$ is assumed to be positive,  then Lemma 2.3 can be strengthened to the following\footnote{See also 
Remark 5.5 in \cite{Cao08b}, as well as Lemma 3.2 in \cite{Chu} .}

\begin{proposition} Let $(M^n, g_{ij}, F)$ be a complete noncompact
gradient steady soliton with positive Ricci curvature $Rc>0$. Assume the scalar curvature $R$ attains its maximum at some origin
$x_0$. Then, there exist some constants $0<c_1\leq \sqrt{C_0}$ and $c_2>0$  such that the potential function $F$ satisfies the estimates
$$c_1r(x)-c_2 \leq F(x) \leq \sqrt{C_0} r(x) + |F(x_0)|, \eqno(2.7)$$
where $r(x)=d(x_0, x)$ is the distance function from  $x_0$, and $C_0=R_{max}$ is the constant in (2.2).
In particular, $F$ is a strictly convex exhaustion function achieving its minimum at the only critical point  $x_0$, and the underlying manifold
$M^n$ is diffeomorphic to $\mathbb{R}^n$.

\end{proposition}

\begin{proof}
It is clear that the upper estimate in (2.7) in fact holds for complete gradient steady solitons in general,
because $|\nabla F|^2\leq C_0$ by (2.3) and Lemma 2.2.

To prove the lower estimate, we consider any minimizing normal geodesic $\gamma(s)$, $0\leq s \leq
s_0$ for large $s_0>0$, starting from the origin
$x_0=\gamma(0)$. Denote by $X(s)=\dot\gamma(s)$, the unit tangent
vector along $\gamma$, and $\dot{F}=\nabla_XF(\gamma(s))$. By (1.1), we have
$$\nabla_X\dot{F}= \nabla_X\nabla_XF= Rc(X,X). \eqno(2.8) $$
Integrating (2.8) along $\gamma$ and noting that $x_0$ is the (unique) minimum point of $F$,
we get, for $s\ge 1$,

\begin{align*}
\dot{F}(\gamma(s))
=\int_0^{s} Rc(X,X)ds \geq \int_0^{1} Rc(X,X)ds \geq c_1,
\end{align*}
where $c_1>0$ is the least eigenvalue of  $Rc$ on the unit geodesic ball $B_{x_0}(1)$.
Thus,
$$ F(\gamma(s_0)) =\int_1^{s_0} \dot{F}(\gamma(s)) ds + F(\gamma(1))\geq c_1s_{0}-c_1 +F(\gamma(1)).$$

\end{proof}

Now we turn our attention to locally conformally flat steady Ricci solitons.  For any Riemannian manifold $(M^n, g)$, let
\begin{align*}
W_{ijkl}  = & R_{ijkl} - \frac{1}{n-2}(g_{ik}R_{jl}-g_{il}R_{jk}-g_{jk}R_{il}+g_{jl}R_{ik})\\
& + \frac{R}{(n-1)(n-2)} (g_{ik}g_{jl}-g_{il}g_{jk}) 
\end{align*}
be the Weyl tensor, and denote by
 $$C_{ijk}=\nabla_k R_{ij}-\nabla_j R_{ik}-\frac {1}{2(n-1)} ( g_{ij} \nabla_k R - g_{ik} \nabla_j R) \eqno(2.9)$$
 the Cotton tensor.  It is well known that, for $n=3$, $W_{ijkl}$ vanishes identically while $C_{ijk}=0$ if and only if
$(M^3, g_{ij})$ is locally conformally flat; for $n\ge 4$, $W_{ijkl}=0$ if and only if  $(M^n, g_{ij})$ is locally
conformally flat.  Moreover, for $n\ge 4$, the vanishing of Weyl tensor $W_{ijkl}$ implies the vanishing of the Cotton
tensor $C_{ijk}$, while $C_{ij k} = 0 $ corresponds to the Weyl
tensor being harmonic.

In proving Theorem 1.2, we need the following important facts due to B.-L. Chen \cite {BChen} (for n=3)
and Z.-H. Zhang \cite{Zh2} (for $n\ge 4$):

\begin{proposition} Let $(M^n, g_{ij}, F)$ be a complete gradient steady Ricci soliton. Then
$(M^n, g_{ij}, F)$ has nonnegative curvature operator $Rm\ge 0$, provided either

(a)  $n=3$,  or

\smallskip

(b)  $n\ge 4$ and  $(M^n, g_{ij}, F)$ is locally conformally flat.

\end{proposition}

\begin{remark} Part (a) is a special case of a more general result due to B.-L. Chen \cite{BChen} for 3-dimensional complete
ancient solutions.  Part  (b) is essentially proved by Z.-H. Zhang \cite{Zh2}. In fact, the same arguments in \cite{Zh2} imply that a
complete ancient solution $g_{ij}(t)$ to the Ricci flow with vanishing Weyl tensor for each time $t$ is necessarily of
nonnegative curvature operator.
\end{remark}

Combining Proposition 2.4 with (2.3) in Lemma 2.1,  we have

\begin{proposition} Let $(M^n, g_{ij}, F)$ be a complete gradient steady soliton such that
 either $n=3$, or  $n\ge 4$ and  $(M^n, g_{ij}, F)$ is locally conformally flat. Then
 $(M^n, g_{ij}, F)$ has bounded  and nonnegative curvature operator $0\le Rm \le C$.

\end{proposition}

\section{The proofs of Theorem 1.1 and Theorem 1.2}

Throughout this section,  we assume that $(M^n, g_{ij}, F)$ ($n\ge
3$) is a complete noncompact locally conformally flat gradient
steady soliton.  We are going to prove Theorem 1.1, which is the same as Proposition 3.1 below, and Theorem 1.2. 

\begin{proposition} If $(M^n, g_{ij}, F)$, $n\ge
3$, is a $n$-dimensional complete noncompact locally conformally flat gradient steady soliton with positive sectional curvature, then  
$(M^n, g_{ij}, F)$ is a rotationally symmetric gradient staedy soliton on $\mathbb{R}^n$, hence isometric to the Bryant soliton.
\end{proposition}

\begin{proof}   First of all, since the sectional curvature of $(M^n, g_{ij}, F)$ is positive we know that $M^n$ is diffeomorphic to $\mathbb{R}^n$ by Gromoll-Meyer \cite{GM} 
(or by Proposition 2.3 if $R$ attains its maximum at some origin $x_0$). Moreover, since the Ricci curvature is positive,  the potential function $F$ is  strictly
convex, thus having at most one critical point. Secondly, if we denote by  $G=|\nabla F|^{2}, $ then in any neighborhood, where  $G\neq 0$, of
the equipotential hypersurface $$\Sigma_c=: \{x\in M: F(x)=c\} $$ of a
regular value $c$ of $F$,  we can express the metric
$ds^2=g_{ij}(x)dx^{i}dx^{j}$ as
$$ds^2=\frac {1}{G(F, \theta)}dF^2  + g_{ab}(F, \theta) d\theta^{a} d\theta^{b},
\eqno(3.1)$$
where $\theta=(\theta^{2}, \cdots, \theta^{n})$ denotes
intrinsic coordinates for $\Sigma_c$. It is clear that the key step in proving Proposition 3.1 is to show that in (3.1) we have $G=G(F)$, $g_{ab}=g_{ab}(F)$, and that ($\Sigma_c,g_{ab}$) is a
space form of positive curvature.   

Note that an  $n$-dimensional rotationally symmetric metric is of the form 
$$ds^2=\psi^2(t)dt^2+\varphi^2(t)\bar{g},$$
where $\bar g$ is the standard metric on the unit sphere $\mathbb{S}^{n-1}$,  
and the Ricci tensor of such a metric is given by
$$Rc=(n-1)(-\frac{\varphi_{tt}}{\varphi}+\frac{\varphi_t\psi_t}{\varphi\psi})dt^2 - (\frac{\varphi\varphi_{tt}}{\psi^2} +\frac{(n-2)\varphi^2_{t}}{\psi^2} - \frac{\varphi\varphi_t\psi_t}{\psi^3}+2-n)\bar g.$$
In particular, it has at most two distinct eigenvalues depending only on $t$.

We shall first derive a useful formula for the norm square of the Cotton tensor for steady solitons with vanishing Weyl tensor.  
This formula plays an important role in our derivation of the desired property of the Ricic tensor for the steady solton metric. Moreover, as $W_{ijkl}=0$ always when $n=3$, 
this formula is  of particular interest in the three-dimensional case.

\begin{lemma}  For any $n$-dimensional gradient steady soliton ($n\ge 3$) with vanishing Weyl tensor $W_{ijkl}=0$, we have 
$$ |C_{ijk}|^2=\frac{1}{(n-2)^2} (|R_{ik} \nabla_j F - R_{ij} \nabla_k F|^2-\frac{2}{(n-1)} |R\nabla F-\frac{1}{2}\nabla G|^2).\eqno(3.2)$$
Here $C_{ijk}$ is the Cotton tensor defined by (2.9).
\end{lemma} 

\begin{proof}
Notice that, by (1.1) and Lemma 2.1, we have $$\nabla G = 2Rc (\nabla F, \cdot)=-\nabla R. \eqno(3.3)$$
Using (2.1), (3.3), and the assumption $W_{ijkl}=0$, we can express 
\begin{align*}
 C_{ijk} =&\nabla_k R_{ij}-\nabla_j R_{ik}-\frac {1}{2(n-1)} ( g_{ij} \nabla_k R - g_{ik} \nabla_j R)\\
= & R_{kjil} \nabla_l F + \frac{1}{2(n-1)} (g_{ij}\nabla_kG -g_{ik}\nabla_j G)\\
             = & \frac{1}{n-2}(R_{ik} \nabla_j F- R_{ij} \nabla_k F) +\frac{1}{2(n-1)(n-2)} (g_{ik}\nabla_j G -g_{ij} \nabla_k G)\\
              & - \frac{R}{(n-1)(n-2)} (g_{ik} \nabla_j F - g_{ij}\nabla_k F).
 \end{align*}
Hence, by direct computations, we have 
\begin{align*}
|C_{ijk}|^2= & \frac{1}{(n-2)^2}  |R_{ik} \nabla_j F - R_{ij} \nabla_k F|^2 +\frac{1}{2(n-1)(n-2)^2} |\nabla G|^2 \\
& +\frac{2R^2}{(n-1)(n-2)^2}|\nabla F|^2 +\frac{2}{(n-1)(n-2)^2}[R\nabla G\cdot\nabla F-Rc(\nabla F, \nabla G)] \\ & 
-\frac{4R} {(n-1)(n-2)^2}[R|\nabla F|^2 - Rc(\nabla F, \nabla F)] -\frac{2R} {(n-1)(n-2)^2} \nabla G\cdot \nabla F.
\end{align*}
On the other hand, by (3.3), $$Rc (\nabla F, \nabla G)=\frac{1}{2}|\nabla G|^2,$$ and $$ Rc (\nabla F, \nabla F)=\frac{1}{2}\nabla G\cdot \nabla F.$$
Thus, 
\begin{align*}
|C_{ijk}|^2= & \frac{1}{(n-2)^2}  |R_{ik} \nabla_j F - R_{ij} \nabla_k F|^2 \\ 
-& \frac{1}{2(n-1)(n-2)^2}(|\nabla G|^2-4R\nabla G\cdot\nabla F+4R^2|\nabla F|^2)\\
=& \frac{1}{(n-2)^2} (|R_{ik} \nabla_j F - R_{ij} \nabla_k F|^2 - \frac{2}{(n-1)} |R\nabla F-\frac{1}{2}\nabla G|^2). 
\end{align*}

\end{proof}

Next, using Lemma 3.1, we show that the Ricci tensor of our $(M^n, g_{ij}, F)$ has, at least pointwisely, the desired property that it 
has at most two distinct eigenvalues. 

\begin{lemma} At any point $p\in \Sigma_c$, the Ricci tensor of $(M^n, g_{ij}, F)$ either has a unique eigenvalue $\lambda$,
or has two distinct eigenvalues  $\lambda$ and $\mu$ of multiplicity $1$ and $n-1$ respectively.
In either case,  $e_1=\nabla F /|\nabla F |$ is an eigenvector with eigenvalue $\lambda$.
In other words, for any orthonormal basis $e_2, \cdots, e_n$ tangent to the level surface $\Sigma_c$ at $p$, the
Ricci tensor has the following properties:

\smallskip

(i) \ \ $Rc(e_1, e_1)=R_{11}$;

(ii) \ $Rc(e_1, e_b)=R_{1b}=0, \ b=2, \cdots, n$;

(iii) $Rc(e_a, e_b)=R_{aa}\delta_{ab}, \  a, b=2, \cdots, n$;

\smallskip

\noindent where either $R_{11}=\cdots=R_{nn}=\lambda$ or $R_{11}=\lambda$ and $R_{22}=\cdots=R_{nn}=\mu$.
\end{lemma}

\begin{proof}  For any regular value  $c$ of $F$, pick an orthonormal basis $\{E_1, \cdots, E_n\}$ of the tangent space $T_p M$ at 
$p\in \Sigma_c= \{x\in M: F(x)=c\} $ so that $Rc(E_i, E_j)=\lambda_i \delta_{ij}$.    Then, we have
$$|R_{ik} \nabla_j F - R_{ij} \nabla_k F|^2=2 \sum_{j=1}^n |\nabla_j F|^2 \sum_{i\neq j}\lambda_i^2,$$ and 
$$|R\nabla F-\frac{1}{2}\nabla G|^2=\sum_{j=1}^n |\nabla_j F|^2 (\sum_{i\neq j}\lambda_i)^2.$$ Plugging the above two identities into 
(3.2) and noticing that $C_{ijk}=0$  by assumption, it follows that
$$\sum_{j=1}^n |\nabla_j F|^2(\sum_{i\neq
j, k\neq j}(\lambda_i-\lambda_k)^2)=0. \eqno(3.4)$$
Since  $c$ is a regular value of $F$ and $p\in \Sigma_c$, $\nabla F(p)\neq 0$.  On the other hand, from (3.4) 
it is easy to see that if $\nabla F(p)$ has two or more nonzero components with respect to  $\{E_i\}_{i=1}^{n}$, then $\lambda_1=\cdots=\lambda_n$
so the Ricci tensor has a unique eigenvalue. Otherwise, say $\nabla_{1} F\neq 0$ and $\nabla_{i} F=0$ for $i=2, \cdots, n$,
then $\nabla F=|\nabla F| E_1$, with $E_1= \nabla F/|\nabla F|$, is an eigenvector of $Rc$, and $\lambda_2=\cdots=\lambda_n$. In either case, we conclude that
$\nabla F$ is an eigenvector, and the Ricci tensor has the desired properties.

\end{proof}

\begin{remark}  Fern\'andez-L\'opez and Garc\'ia-R\'io \cite{FG} showed, by a different argument, that for shrinking solitons
with harmonic Weyl tensor the gradient of any potential function is an eigenvector of
the Ricci tensor.
\end{remark}

Finally we show that the eigenvalues of the Ricci tensor, and the scalar curvature $R$ are constant on level surfaces $\Sigma_c=\{F=c\}$.

\begin{lemma} Let $c$ be a regular value of $F$ and $\Sigma_c=\{F=c\}$. Then:

 \medskip

(a) The function $G = |\nabla F|^2$ is constant on $\Sigma_c$, i.e., $G$ is a function of $F$ only;

\smallskip

(b) The scalar curvature $R$ of $(M^n, g_{ij}, F)$ is constant on $\Sigma_c$;

\smallskip

(c) The second fundamental form $h_{ab}$ of $\Sigma_c$ is of the form $h_{ab}=\frac{H}{n-1} g_{ab}$;

\smallskip

(d) The mean curvature $H$ is constant on $\Sigma_c$;

\smallskip

(e) $\Sigma_c$, with the induced metric $g_{ab}$,  is of constant sectional curvature.

\end{lemma}

\begin{proof} Let  $\{e_1, e_2, \cdots, e_n\}$ be an orthonormal frame with $e_1=\nabla F/|\nabla F|$ and $e_2, \cdots, e_n$
tangent to $\Sigma_c$.

First of all, by (3.3) and Lemma 3.1 (ii),
$$ \nabla_a G=2Rc(\nabla F, e_a)=0, \  a=2, \cdots, n.$$ Hence $G$ is constant on $\Sigma_c$,
 and so is $R$ by (2.3). This proves (a) and (b).

Next, since $R_{ab}=R_{aa} g_{ab}$ and $R_{22}=\cdots=R_{nn}$ by Lemma 3.1, we  have
$$h_{ab}=G^{-1/2}R_{ab}=\frac{H}{n-1} g_{ab}, \ a, b=2, \cdots, n,$$ where  $$ H=G^{-1/2} (R-R_{11}).$$
Moreover, the Codazzi equation says that, for $a, b, c=2, \cdots, n$,
$$R_{1cab}=\nabla_a^{\Sigma_c}h_{bc}-\nabla_b^{\Sigma_c}h_{ac}.$$ Tracing over b and c, we obtain
$$0=R_{1a}=\nabla_a^{\Sigma_c} H- \nabla_b^{\Sigma_c}h_{ab}=(1-\frac{1}{n-1})e_a H,$$ proving (c) and (d).

Finally, by the Gauss equation, the sectional curvatures of
$(\Sigma_c, g_{ab})$ are given by

$$R^{\Sigma_c}_{abab} = R_{abab} +h_{aa}h_{bb}-h_{ab}^2= R_{abab} + \frac{H^2}{(n-1)^2}. \eqno(3.5)$$
On the other hand, since $W_{ijkl}=0$, 
\begin{align*}
R_{abab}=  & \frac{2}{n-2} R_{aa}-\frac{R}{(n-1)(n-2)} \\
= &  \frac{2 G^{1/2}H-R}{(n-1)(n-2)}.
\end{align*}
But we already showed that $G, H, R$ are constant on $\Sigma_c$.  Therefore $(\Sigma_c, g_{ab})$ has constant sectional curvature
$$K_c=  \frac{2 G^{1/2}H-R}{(n-1)(n-2)}+ \frac{H^2}{(n-1)^2},$$ proving (e).
\end{proof}

With Lemma 3.3 in our hands, we now conclude the proof of Proposition 3.1: Recall that in any neighborhood, where  $G\neq 0$, of
the equipotential hypersurface $$\Sigma_c=: \{x\in M: F(x)=c\} $$ of a
regular value $c$ of $F$,  we can express the metric
$ds^2=g_{ij}(x)dx^{i}dx^{j}$ as
$$ds^2=G^{-1}(F, \theta) dF^2 + g_{ab}(F, \theta) d\theta^{a} d\theta^{b},$$
where $\theta=(\theta^{2}, \cdots, \theta^{n})$ denotes
intrinsic coordinates for $\Sigma_c$. Then Lemma 3.3 tells us
that $G=G(F)$, $g_{ab}=g_{ab}(F)$, and that ($\Sigma_c,g_{ab}$) is a
space form,  with positive curvature (this follows from (3.5) and the assumption that $(M^n,
g_{ij}, F)$ has positive sectional curvature). Also, F has exactly one minima
at some origin $x_0$, otherwise $(M^n, g_{ij}, F)$ would split out a flat factor which is impossible.
Hence, on $M \setminus \{x_0\}$ we have $$ds^2=G^{-1}(F) dF^2 + \varphi^2(F)\bar{g}, \eqno(3.6)$$
where $\bar g$ denotes the standard metric on the unit sphere
$\mathbb{S}^{n-1}$, and $\phi$ is some smooth function on $M^n$
depending only on $F$ and vanishing only at $x_0$. Thus, $(M^n,
g_{ij}, F)$ is a rotationally symmetric gradient steady soliton on
$\mathbb{R}^n$. Therefore, it is the Bryant soliton.
\end{proof}

\medskip

\noindent {\bf Proof of Theorem 1.2}

\begin {proof}
Now, by assumption,  $(M^n, g_{ij}, F)$ is a complete noncompact locally conformally
flat gradient steady soliton. By  Proposition 2.5, we know
that  $(M^n, g_{ij}, F)$ has bounded and nonnegative curvature
operator $0\le Rm\le C$.  From Hamilton's  strong maximum principle (see
\cite{Ha86} and \cite {Ha95F}),  $(M^n, g_{ij}, F)$ is either of
positive curvature operator, or its holonomy group reduces.  

If   $(M^n, g_{ij}, F)$ has positive curvature operator $Rm>0$, then by Proposition 3.1/Theorem 1.1, $(M^n, g_{ij}, F)$ must be the Bryant soliton.

On the other hand, if $(M^n, g_{ij}, F)$ has reduced holonomy, then $(M^n, g_{ij}, F)$  is either a Riemannian product, or a locally symmetric 
space, or irreducible but not locally symmetric. 

\smallskip
$\bullet$ Case (a): $(M^n, g_{ij}, F)$ is a Riemannian product.

It is known (cf. p.61 in \cite{Besse}) that the only conformally flat Riemannian products are
either the product of a space form $N^{n-1}$ with $\mathbb{S}^1$ or $\mathbb{R}^1$, or the product of two Riemannian manifolds,
one with constant sectional curvature $-1$ and the other with constant sectional curvature $1$. However, since our $(M^n, g_{ij})$ is a steady
Ricci soliton with nonnegative sectional curvature,  this implies that the latter case cannot occur and that in the former case $N^{n-1}$ must be flat. Hence
$(M^n, g_{ij}, F)$ is flat in case (a);

\smallskip
$\bullet$ Case (b):  $(M^n, g_{ij}, F)$ is a locally symmetric space.

In case (b), $(M^n, g_{ij}, F)$ is necessarily Einstein. However, $(M^n, g_{ij}, F)$ is noncompact and of nonnegative sectional curvature. Therefore it 
follows that it must be Ricci flat and hence flat;

\smallskip
$\bullet$ Case (c):  $(M^n, g_{ij}, F)$ is irreducible and not locally symmetric.

In this case the flatness of $(M^n, g_{ij}, F)$ follows from the holonomy classification theorem of Berger and Simons (cf. \cite{Besse}),
the fact that $$|W_{ijkl}|^2\ge \frac{3(m-1)}{m(m+1)(2m-1)}R^2 \eqno(3.7)$$ for K\"ahler manifolds of complex dimension $m>1$
(cf. Proposition 2.68 in \cite{Besse}), and that our $(M^n, g_{ij}, F)$ is noncompact, conformally flat and of nonnegative sectional curvature.

Thus we have shown that in all above three cases $(M^n, g_{ij}, F)$ are flat. This completes the proof of Theorem 1.2. 
\end{proof}

\smallskip
\begin{remark} Shortly after our work appeared on the arXiv, Catino and Mantegazza \cite{CM} obtained a similar result to our Theorem 1.1 for dimension 
$n\ge 4$ by studying the evolution of Weyl tensor under the Ricci flow.  However, as pointed out in \cite{CM}, their argument does not work for $n=3$.
\end{remark}

\section{Further Remarks}

In this section, we point out that several lemmas in the proof of Theorem 1.1 for gradient steady Ricci solitons also hold  for gradient shrinking and expanding
Ricci solitons satisfying
$$R_{ij}=\nabla_i\nabla_j F + \rho g_{ij}, \eqno(4.1)$$ with $\rho=1/2$ for shrinkers and  $\rho=-1/2 $ for expanders respectively. In particular, our method, when combined with a 
result of Kotschwar \cite{Kot}, yields another proof of the classification theorem for n-dimensional ($n\ge 4$) complete locally conformally flat shrinking gradient 
solitons (see Proposition 4.1). 

\medskip

\begin{remark}  A similar version of Lemma 3.1 is valid for locally conformally flat gradient shrinkers and expanders. 
 
First of all, similar arguments as in showing Lemma 2.1 imply that 
$$\nabla_iR=-2R_{ij}\nabla_jF, $$ and
$$R+|\nabla F|^2+2\rho F=C_0 $$ for some constant $C_0$. Hence, by defining
$$G_{\rho}=: G+2\rho F=|\nabla F|^2+2\rho F, $$  we have 
$$R+G_{\rho}=C_0 $$ and $$ \nabla G_{\rho} =-\nabla R= 2Rc (\nabla F, \cdot). \eqno(4.2)$$

Therefore, by replacing $G=|\nabla F|^2$ by $G_{\rho}=G+2\rho F$ and carrying out the same arguments as in the proof of Lemma 3.1, we 
obtain 

\begin{lemma}  For any $n$-dimensional gradient shrinking or expanding soliton ($n\ge 3$) with vanishing Weyl tensor $W_{ijkl}=0$, the Cotton tensor
$C_{ijk}$ has the property that 
$$ |C_{ijk}|^2=\frac{1}{(n-2)^2} (|R_{ik} \nabla_j F - R_{ij} \nabla_k F|^2-\frac{2}{(n-1)} |R\nabla F-\frac{1}{2}\nabla G_{\rho}|^2).\eqno(4.3)$$
\end{lemma} 

Consequently, Lemma 3.2 and Lemma 3.3 hold for locally conformally flat gradient shrinkers 
and expanders as well.
\end{remark}

\begin{lemma}  Let $(M^n, g_{ij}, F)$ be any $n$-dimensional, $n\ge 3 $, complete locally conformally flat gradient shrinking or expanding Ricci soliton
satisfying (4.1). Then, at any regular point $p\in \Sigma_c$, the
Ricci tensor of $(M^n, g_{ij}, F)$ either has a unique eigenvalue
$\lambda$, or has two distinct eigenvalues  $\lambda$ and $\mu$ of
multiplicity $1$ and $n-1$ respectively. In either case,
$e_1=\nabla F /|\nabla F |$ is an eigenvector with eigenvalue
$\lambda$. In other words, for any orthonormal basis $e_2, \cdots,
e_n$ tangent to the level surface $\Sigma_c$ at $p$, the Ricci
tensor has the following properties:

\smallskip

(i) \ \ $Rc(e_1, e_1)=R_{11}$;

(ii) \ $Rc(e_1, e_b)=R_{1b}=0, \ b=2, \cdots, n$;

(iii) $Rc(e_a, e_b)=R_{aa}\delta_{ab}, \  a, b=2, \cdots, n$;

\smallskip

\noindent where either $R_{11}=\cdots=R_{nn}=\lambda$ or $R_{11}=\lambda$ and $R_{22}=\cdots=R_{nn}=\mu$.
\end{lemma}

\begin{lemma} Let $(M^n, g_{ij}, F)$ be any $n$-dimensional, $n\ge 3 $, complete locally conformally flat gradient shrinking or expanding Ricci soliton
satisfying (4.1), and let  $c$ be a regular value of $F$ and $\Sigma_c=\{F=c\}$. Then:

 \medskip

(a) The function $G = |\nabla F|^2$ is constant on $\Sigma_c$, i.e., $G$ is a function of $F$ only;

\smallskip

(b) $R$, the scalar curvature of $(M^n, g_{ij}, F)$, is constant on $\Sigma_c$;

\smallskip

(c) The second fundamental form $h_{ab}$ of $\Sigma_c$ is of the form $h_{ab}=\frac{H}{n-1} g_{ab}$;

\smallskip

(d) The mean curvature $H$ is constant on $\Sigma_c$;

\smallskip

(e) $(\Sigma_c, g_{ab})$  is of constant sectional curvature $K_{c,\rho}$.

\end{lemma}

Note that for shrinking solitons, we have $K_{c,1/2} \geq 0$ according to \cite{Zh2}.

\smallskip

\begin{remark}  There is another formula for $|C_{ijk}|^2$ which is valid for all $n$-dimensional ($n\ge 3
$) complete gradient {\sl shrinking, steady}, and {\sl expanding} Ricci
solitons with vanishing Weyl tensor. This new formula relates $|C_{ijk}|^2$ more explicitly to the geometry of the 
level surfaces of $F$, and immediately implies Lemma 3.2/4.2 and Lemma 3.3/4.3 (a)-(d). 

\begin{lemma} For any $n$-dimensional gradient shrinking, or steady, or expanding soliton $(M^n, g_{ij}, F)$ ($n\ge 3$) with vanishing Weyl tensor $W_{ijkl}=0$, we have 
$$|C_{ijk}|^2=\frac{2G^2}{(n-2)^2} |h_{ab}-\frac{H}{n-1}g_{ab}|^2 +\frac{1}{2(n-1)(n-2)}
|\nabla_a G|^2, \eqno(4.4)$$ 
where $h_{ab}$ and $H$ are, as in Lemma 3.3, the second fundamental form and the mean curvature for the level surface 
$\Sigma_c=\{F=c\}$ at any regular value $c$ of $F$. 
\end{lemma} 

\begin{proof} Denote again $$G=|\nabla F|^2  \quad \mbox{and} \quad G_{\rho}=G+2\rho F$$ so that $G_0=G$. 
Let $c$ be any regular value of the potential function $F$ and consider the corresponding  
equipotential hypersurface $$\Sigma_c=: \{x\in M: F(x)=c\}. $$ Let  $\{e_1, e_2, \cdots, e_n\}$ be any orthonormal frame, with 
$e_1=\nabla F/|\nabla F|=\nabla F/\sqrt{G}$ and $e_2, \cdots, e_n$ tangent to $\Sigma_c$. Then the second fundamental form
$h_{ab}$ and the mean curvature $H$ are given respectively  by 
\begin{align*}
h_{ab}=& <\nabla_a e_1, e_b>
= <\nabla_a \frac{\nabla F} {\sqrt{G}}, e_b>\\
=& \frac{1}{\sqrt{G}}R_{ab}, \quad a,b=2,\cdots, n
\end{align*}
 and 
$$H=\frac{1}{\sqrt{G}}(R-R_{11}).$$  

By (4.2), $$R_{11}=\frac{1}{G} Rc (\nabla F, \nabla F)=\frac{1}{2G}\nabla F\cdot\nabla G_{\rho}$$
and 
$$R_{1a}=\frac{1} {\sqrt{G}} Rc( {\nabla F}, e_a)=\frac{1} {2\sqrt{G}}\nabla_a G_{\rho}$$
Moreover, 
\begin{align*}
|\nabla_{\Sigma_c}G_{\rho}|^2= & \sum_{a=2}^n|\nabla_a G_{\rho}|^2\\
= & |\nabla G_{\rho}|^2 -\frac{1}{G} |\nabla G_{\rho}\cdot\nabla F|^2.
\end{align*}
Thus, by direct computations, we obtain
\begin{align*}
|h_{ab}-\frac{H}{n-1}g_{ab}|^2= & |h_{ab}|^2-\frac{H^2}{n-1}\\
=& \frac{1}{G} (|Rc|^2-2\sum_{a=2}^nR^2_{1a}-R^2_{11})-\frac{(R-R_{11})}{(n-1)G}^2\\
=& \frac{1}{G} |Rc|^2 + \frac{R}{(n-1)G^2}\nabla F\cdot\nabla G_{\rho}-\frac{1}{2G^2}|\nabla G_{\rho}|^2 \\
& +\frac{n-2}{4(n-1)G^3}|\nabla F\cdot\nabla G_{\rho}|^2-\frac{R^2}{(n-1)G},
\end{align*}
Now noticing that, by Lemma 4.1, we also have  
\begin{align*}
 |C_{ijk}|^2=& \frac{2G}{(n-2)^2} |Rc|^2-\frac{1}{2(n-2)^2} |\nabla G_{\rho}|^2\\
 -& \frac{1} {2(n-1)(n-2)^2} (|\nabla G_{\rho}|^2-4R\nabla G_{\rho} \cdot \nabla F+4R^2G).
\end{align*}
Therefore, one can verify directly that 
\begin{align*}
\frac{2G^2}{(n-2)^2}|h_{ab}-\frac{H}{n-1}g_{ab}|^2= |C_{ijk}|^2-\frac{1}{2(n-1)(n-2)}
|\nabla_{\Sigma_c}G_{\rho}|^2.
\end{align*}
This completes the proof of lemma 4.4.

\end{proof}

\end{remark}

\smallskip

\begin{remark} Based on Lemma 4.3 for shrinkers, it is not hard to see that the universal cover  $(\tilde{M}^n,  \tilde{g}_{ij}, \tilde{F})$ of a complete 
locally conformally flat gradient shrinking Ricci soliton is a rotationally symmetric gradient shrinking Ricci soliton and  $\tilde{M}^n$ is diffeomorphic to either
$\mathbb{R}^n$, or $\mathbb{S}^n$, or $\mathbb{S}^{n-1}\times \mathbb{R}$.  Combining this with the result of 
Kotschwar \cite{Kot},  one obtains another proof of the following classification theorem for n-dimensional ($n\ge 4$) complete locally conformally flat shrinking gradient 
solitons (which is first due to the combined works of Z.-H. Zhang  and Ni-Wallach, and also the combination of more recent  work of 
Munteanu-Sesum and the earlier works of either Petersen-Wylie, or X. Cao, B. Wang and Z. Zhang). 

\begin{proposition} Any $n$-dimensional, $n\ge 4$, complete locally conformally flat gradient shrinking Ricci soliton is a finite quotient of
$\mathbb{R}^n$, or $\mathbb{S}^n$, or $\mathbb{S}^{n-1}\times \mathbb{R}$. 
\end{proposition}

\end{remark}

\end{document}